\documentclass[11pt,a4paper]{article}
\usepackage[T1]{fontenc}
\usepackage{lmodern}
\usepackage[margin=27mm]{geometry}
\usepackage{amsmath,amssymb,amsthm,mathtools}
\usepackage{microtype}
\usepackage[hidelinks]{hyperref}
\usepackage{enumitem}
\setlist[enumerate]{itemsep=3pt,topsep=5pt}
\allowdisplaybreaks[2]
\newtheorem{theorem}{Theorem}[section]
\newtheorem{lemma}[theorem]{Lemma}
\newtheorem{proposition}[theorem]{Proposition}

\theoremstyle{remark}
\newtheorem{remark}[theorem]{Remark}
\newcommand{\R}{\mathbb R}
\newcommand{\E}{\mathbb E}
\newcommand{\Pp}{\mathbb P}
\newcommand{\F}{\mathcal F}
\newcommand{\G}{\mathcal G}
\newcommand{\ind}{\mathbf 1}
\newcommand{\st}{\xrightarrow{\mathcal L\text{-}s}}

\newcommand{\norm}[1]{\lVert #1\rVert}

\DeclareMathOperator{\Cov}{Cov}

\title{A fixed-time stable limit theorem for non-normalized\\
functionals on irregular time grids}
\author{Yi Guo}
\date{}

\begin{document}
\maketitle
\vspace{-1.5em}
\begin{abstract}
Let an It\^o semimartingale be observed at times satisfying
$\tau_{i+1}^n-\tau_i^n=(n\theta_{\tau_i^n})^{-1}$, where $\theta$ is
adapted and c\`adl\`ag, and $\theta$ and its reciprocal are bounded on
compact intervals. Under the usual hypothesis (H) on the semimartingale,
we prove fixed-time stable convergence for test functions whose Hessian
is $o(\norm{x})$ at zero. The limit is a sum of jump contributions with
sampling scale $\theta_{T-}^{-1}$ and separate pre-jump and post-jump
volatilities. The proof establishes joint stable convergence of the local
Poisson positions and Brownian increments on the endogenous grid, and
controls infinitely many small jumps by a uniform It\^o estimate.
Three examples identify the left-limit sampling scale and show why
positivity and c\`adl\`ag regularity alone cannot replace local control
of the reciprocal sampling intensity.
\end{abstract}

\section{Introduction}\label{sec:introduction}
Let $X$ be an It\^o semimartingale on a filtered probability space.
We study non-normalized functionals of its increments, which on an
equidistant grid take the form
\[
 \sum_{i=1}^{\lfloor t/\Delta_n\rfloor}
 f\bigl(X_{i\Delta_n}-X_{(i-1)\Delta_n}\bigr),
 \qquad \Delta_n=\frac1n.
\]
Laws of large numbers and central limit theorems for such functionals,
under suitable assumptions on $X$ and $f$, are developed in
Jacod \cite{jacodrpv} and Jacod and Protter \cite{JP}.
Equidistant observations will be referred to as regular sampling.
The asymptotic theory of high-frequency statistics and discretization
errors under regular sampling is well developed; see, for example,
\cite{ait14,al,barn11,jacodrpv,JP}.

In many applications, observations are not available at equally spaced
times. A regular grid may therefore fail to reflect the observation
mechanism. Irregular sampling also allows observations to be taken more
frequently during selected periods or when the process enters particular
regions of the state space. These considerations motivate the study of
irregular sampling and its effect on discretization errors.

The effect of random sampling intervals on high-frequency asymptotics was studied by A\"it-Sahalia and Mykland \cite{aitmyk03,aitmyk04}. 
Here we consider the recursive sampling scheme studied by Lindberg and
Rootz\'en \cite{lindberg2013error}:
\[
 \tau_0^n=0,\qquad
 \tau_{i+1}^n=\tau_i^n+\frac1{n\theta_{\tau_i^n}}.
\]
The positive adapted process $\theta$ determines the local sampling
frequency: larger values produce shorter observation intervals.
The length of each interval is fixed using the information available
at its left endpoint, so the grid may depend on the observed process.
Lindberg and Rootz\'en establish limit theorems for random-grid
approximations of stochastic integrals, with the sampling intensity
entering the limiting distribution. Related results concern stochastic
integral approximations under irregular sampling in Zhou and Su
\cite{szg}, and irregular discretization in multilevel Monte Carlo
methods in Guo et al.\ \cite{guoyi}.
More general frameworks for irregular sampling are studied by Hayashi,
Jacod and Yoshida \cite{hayaclt} and by Jacod and Protter
\cite[Chapter 14]{JP}.

Another important class of sampling schemes uses hitting times; see
Fukasawa \cite{fuka desi,fuka cltrv,fuka rvss} and Fukasawa and Rosenbaum
\cite{fukarose}. In a hitting-time scheme, the next observation time is
determined by the subsequent evolution of the process. This differs
from the recursion above, for which the next observation time is already
known at the current one. The proof below uses this latter property.

Under hypothesis (H), we assume that $\theta$ has c\`adl\`ag paths and
that both $\theta$ and $1/\theta$ are bounded on compact time intervals.
For $C^2$ test functions satisfying $f(0)=Df(0)=0$ and
\[
 \norm{D^2f(x)}=o(\norm{x})\qquad(x\to0),
\]
we establish stable convergence of the centered and scaled
non-normalized functionals at each fixed deterministic time, jointly
at any finite collection of such times. The contribution of a jump at
$T$ involves the sampling scale $\theta_{T-}^{-1}$ and the two volatility
values $\sigma_{T-}$ and $\sigma_T$. Common jumps of $X$, $\sigma$ and
$\theta$ are allowed. When $\theta\equiv1$, the result recovers the
fixed-time conclusion of \cite[Theorem 5.1.2]{JP}. Our theorem thus
extends that conclusion to the recursive adaptive grid considered here.
It does not assert functional convergence with actual-time jump centering.

Three examples clarify the roles of the sampling assumptions. The
first uses a bounded intensity with a common jump to show that replacing
$\theta_{T-}$ by $\theta_T$ changes the limiting variance. The second
uses a positive adapted intensity that is discontinuous at the origin
to show that a vanishing mesh alone need not make the continuous
contribution negligible at the central limit scale. The third retains
strict positivity and c\`adl\`ag paths, but allows a left limit of
$\theta$ to vanish. It shows that the local bound on $1/\theta$ cannot
simply be omitted while retaining the other assumptions.

The proof follows the localization, jump truncation and Taylor-expansion
framework of Jacod and Protter \cite[Sections 4.4 and 5.1]{JP}.
The adaptation to the present grid requires joint stable convergence
of the local Poisson positions and Brownian increments. We establish
this convergence by conditioning at the left endpoint of each sampling
interval and testing against bounded variables measurable with respect
to the original sigma-field. An adaptive-grid version of the local
increment expansion then identifies the contribution of each isolated
jump. Finally, a uniform It\^o estimate controls the small-function
remainder and permits removal of the cutoff near zero without a
finite-activity assumption.

The paper is organized as follows. Section~\ref{sec:main} introduces
the notation and states the main theorem and its conditional limit law.
Section~\ref{sec:sampling} presents the three examples.
Section~\ref{sec:marks} establishes joint stable convergence of the local
marks after localization. Section~\ref{sec:lemmas} gives the auxiliary
estimates, and Section~\ref{sec:proof} completes the proof of the theorem.

\section{Notation and main result}\label{sec:main}

\subsection{Model and assumptions}

Let $(\Omega,\F,(\F_t)_{t\ge0},\Pp)$ satisfy the usual conditions.
On this space let $W$ be a $d'$-dimensional Brownian motion and let $p$
be a Poisson random measure on $\R_+\times E$ with compensator
$q(ds,dz)=ds\,\lambda(dz)$. These are Brownian and Poisson driving
processes with respect to the same filtration, as in the usual
Grigelionis representation. In particular, for a deterministic set
$A$ with $\lambda(A)<\infty$, $(W,p|_{\R_+\times A})$ has the standard
joint independent-increment law relative to that filtration. This also
holds after a stopping time. Equivalently, this property follows by
applying the exponential-martingale characterization to the Brownian
coordinates and counts on disjoint finite-intensity mark sets.

Consider the $\R^d$-valued process
\begin{equation}\label{eq:X}
 X_t=X_0+\int_0^t b_s\,ds+\int_0^t\sigma_s\,dW_s
 +(\delta\ind_{\{\norm{\delta}\le1\}})\star(p-q)_t
 +(\delta\ind_{\{\norm{\delta}>1\}})\star p_t.
\end{equation}
The coefficients have the usual measurability required in this
representation; in particular $\delta$ is predictable in $(\omega,s)$.
We impose the following assumptions.
\begin{enumerate}[label=\textnormal{(\roman*)}]
 \item \textbf{Semimartingale hypothesis (H).}
 The drift $b$ is locally bounded, and $\sigma$ is adapted and
 c\`adl\`ag. There are stopping times $\rho_m\uparrow\infty$ and
 deterministic functions $\Gamma_m:E\to\R_+$ such that
 \[
  \norm{\delta(\omega,s,z)}\wedge1\le\Gamma_m(z)
  \quad(s\le\rho_m(\omega)),\qquad
  \int_E\Gamma_m(z)^2\lambda(dz)<\infty.
 \]
 Here $\sigma$ denotes the c\`adl\`ag version of the Brownian
 coefficient, so $\sigma_T$ is its post-jump value.
 \item \textbf{Sampling intensity.}
 The process $\theta$ is adapted and c\`adl\`ag and, almost surely,
 for every $H<\infty$,
 \begin{equation}\label{eq:theta}
  0<\inf_{0\le s\le H}\theta_s
  \le\sup_{0\le s\le H}\theta_s<\infty.
 \end{equation}
 \item \textbf{Test function.}
 The function $f:\R^d\to\R$ is $C^2$ and satisfies
 \begin{equation}\label{eq:f}
  f(0)=0,\qquad Df(0)=0,\qquad
  \norm{D^2f(x)}=o(\norm{x})\quad(x\to0).
 \end{equation}
\end{enumerate}
Hypothesis (H) is the standard one used for the regular-grid theorem
in Jacod--Protter \cite[Sections 4.4 and 5.1]{JP}.
Neither $\theta$ nor $\sigma$ is required to be a semimartingale.

\subsection{Statistic and stable limit}

Define
\begin{equation}\label{eq:grid}
 \tau_0^n=0,\qquad
 \tau_{i+1}^n=\tau_i^n+\frac1{n\theta_{\tau_i^n}},\qquad
 N_n(t)=\max\{i:\tau_i^n\le t\},\qquad
 \eta_n(t)=\tau_{N_n(t)}^n.
\end{equation}
Put
\begin{equation}\label{eq:V}
 V_t^n(f)=\sum_{i=1}^{N_n(t)}f(X_{\tau_i^n}-X_{\tau_{i-1}^n}),
 \qquad F_t(f)=\sum_{0<s\le t}f(\Delta X_s).
\end{equation}
Let $(T_p)_{p\ge1}$ exhaust the nonzero jumps of $X$, without repetition.
On a product extension take mutually independent variables
\[
 U_p\sim\mathrm{Unif}(0,1),\qquad
 \Psi_p^-,\Psi_p^+\sim N(0,I_{d'}),\qquad p\ge1,
\]
independent of $\F$. Set
\begin{align}\label{eq:R}
 R_p^{\mathrm{corr}}
 &=\frac1{\sqrt{\theta_{T_p-}}}
   \left(\sqrt{U_p}\,\sigma_{T_p-}\Psi_p^-
       +\sqrt{1-U_p}\,\sigma_{T_p}\Psi_p^+\right),\\
 Z_t(f)&=\sum_{p:T_p\le t}
       Df(\Delta X_{T_p})R_p^{\mathrm{corr}}.\label{eq:Z}
\end{align}
The derivative $Df$ is a row vector. Throughout the paper, stable
convergence is understood relative to the full original sigma-field
$\F$. More precisely, let $Y_n$ be random elements of a Polish space
$\mathsf S$, defined on $(\Omega,\F,\Pp)$, and let $Y$ be an
$\mathsf S$-valued random element on an extension
$(\widetilde\Omega,\widetilde\F,\widetilde\Pp)$.
We write $Y_n\st Y$ if
\[
 \E[B\varphi(Y_n)]\longrightarrow
 \widetilde\E[B\varphi(Y)]
\]
for every bounded $\F$-measurable random variable $B$ and every bounded
continuous function $\varphi:\mathsf S\to\R$, where $B$ is identified
with its lift to the extension. See Jacod and Shiryaev \cite{Jeanshi}
for further background.

\begin{theorem}\label{thm:main}
Under assumptions \textnormal{(i)--(iii)}, the series in
\eqref{eq:V} defining $F_t(f)$ and in \eqref{eq:Z} are absolutely
convergent almost surely for every finite $t$. For each deterministic
$t\ge0$,
\begin{equation}\label{eq:main}
 \sqrt n\bigl(V_t^n(f)-F_t(f)\bigr)\st Z_t(f).
\end{equation}
The convergence also holds jointly at any finite collection of
deterministic times. No independence between $\theta$ and $X$ is required,
and common jumps of $X$, $\sigma$, and $\theta$ are allowed.
\end{theorem}

\begin{remark}
The statement concerns random variables at fixed deterministic times.
It does not assert convergence in the Skorokhod space with the centering
$F_t(f)$. The distinction between actual-time and discretized centering
already occurs on a regular grid \cite[Remark 5.1.3]{JP}.
\end{remark}

\begin{remark}
When $\theta\equiv1$, both the statistic and its limit reduce to those
in the fixed-time statement of \cite[Theorem 5.1.2]{JP}.
If $X$ is continuous, the jump sum is empty and the limit is zero.
\end{remark}

\subsection{Conditional law of the limit}

Let $c_s=\sigma_s\sigma_s^\top$ and write
$v_p=Df(\Delta X_{T_p})$. The conditional law of the theorem is fully
specified by \eqref{eq:R}--\eqref{eq:Z}. In particular,
\begin{align}\label{eq:variance}
 \widetilde\E[Z_t(f)\mid\F]&=0,\notag\\
 \widetilde\E[Z_t(f)^2\mid\F]
 &=\sum_{p:T_p\le t}
       \frac{v_p(c_{T_p-}+c_{T_p})v_p^\top}{2\theta_{T_p-}}.
\end{align}
Indeed, conditional on $\F$ and all the $U_p$, the two Brownian marks
at a jump are independent and its conditional covariance is
\[
 \Cov(R_p^{\mathrm{corr}}\mid\F,U_p)
       =\frac{U_p c_{T_p-}+(1-U_p)c_{T_p}}{\theta_{T_p-}}.
\]
Integration over $U_p$, followed by summation over independent marks,
gives \eqref{eq:variance}. If $c_{T_p-}=c_{T_p}$ at every nonzero jump,
the conditional law given $\F$ is Gaussian. Otherwise it is generally
a Gaussian mixture; its variance alone does not specify that law.

\section{The role of the sampling assumptions}\label{sec:sampling}

The three examples address distinct points. Proposition~\ref{prop:left-scale}
identifies the sampling value that enters the limit while satisfying
all the theorem's assumptions. Proposition~\ref{prop:unbounded-reciprocal}
shows that positivity, adaptedness, and a vanishing mesh are insufficient.
Proposition~\ref{prop:cadlag-reciprocal} retains c\`adl\`ag paths and
violates only the local bound on $1/\theta$, demonstrating that this
bound cannot be omitted from the stated theorem without a replacement
condition.

\begin{samepage}
\subsection{The sampling intensity at a jump}

\begin{proposition}[Left-limit sampling scale]\label{prop:left-scale}
Let $N$ be a rate-one Poisson process, let $T$ be its first jump time,
and let $W$ be a Brownian motion independent of $N$. On their usual
joint filtration, set
\[
 J_s=\ind_{\{T\le s\}},\qquad X_s=W_s+J_s,\qquad
 \theta_s=1+J_s,\qquad f(x)=x^4.
\]
Then $X$, $\theta$, and $f$ satisfy assumptions \textnormal{(i)--(iii)}.
For every deterministic $t>0$,
\begin{equation}\label{eq:left-example}
 \sqrt n\bigl(V_t^n(f)-J_t\bigr)
 \xrightarrow{\mathcal L}4Z\ind_{\{T\le t\}},
\end{equation}
where $Z\sim N(0,1)$ is independent of $T$. Replacing
$\theta_{T-}$ by $\theta_T$ in \eqref{eq:R} gives a different law.
\end{proposition}
\end{samepage}

\begin{proof}
The representation
\[
 J_s=\int_0^s(1-J_{u-})\,dN_u
     =\int_0^s(1-J_{u-})\,du
       +\int_0^s(1-J_{u-})\,d(N_u-u)
\]
verifies (H), with bounded coefficients and Brownian coefficient
$\sigma\equiv1$. Moreover, $1\le\theta\le2$, $\theta$ is adapted and
c\`adl\`ag, and $f''(x)=12x^2=o(|x|)$.

Almost surely, $nT$ is not an integer for any $n$. Until the first grid
point after $T$, all steps have length $1/n$. Thus the interval
containing $T$ has endpoints
\[
 L_n=\frac{\lfloor nT\rfloor}{n},\qquad
 R_n=\frac{\lfloor nT\rfloor+1}{n},\qquad R_n-L_n=\frac1n.
\]
Subsequent steps have length $1/(2n)$. Conditional on $T$, the whole
grid is deterministic and $W$ remains Brownian. Write
\[
 G_n=W_{R_n}-W_{L_n},\qquad
 \xi_i^n=W_{\tau_i^n}-W_{\tau_{i-1}^n},\qquad
 h_i^n=\tau_i^n-\tau_{i-1}^n.
\]
Expanding the fourth power on the interval containing $T$ gives
\begin{align}\label{eq:left-expansion}
 \sqrt n\bigl(V_t^n(f)-J_{\eta_n(t)}\bigr)
 &=\sqrt n
   \sum_{\substack{1\le i\le N_n(t)\\
                    T\notin(\tau_{i-1}^n,\tau_i^n]}}
        (\xi_i^n)^4
   +4\sqrt n\,G_n\ind_{\{R_n\le t\}}\notag\\
 &\quad+\sqrt n\bigl(6G_n^2+4G_n^3+G_n^4\bigr)
                         \ind_{\{R_n\le t\}}.
\end{align}
The first term is nonnegative, and its conditional expectation is at
most
\[
 3\sqrt n\sum_{i=1}^{N_n(t)}(h_i^n)^2
 \le 3\sqrt n\,\frac1n\sum_{i=1}^{N_n(t)}h_i^n
 \le\frac{3t}{\sqrt n}.
\]
Also $G_n\mid T\sim N(0,1/n)$, so the last term of
\eqref{eq:left-expansion} tends to zero in $L^1$.
Since $0<R_n-T\le1/n$ and the density of $T$ is bounded by one,
\[
 \Pp(T\le t<R_n)\le
 \Pp\bigl((t-1/n)^+<T\le t\bigr)\le\frac1n.
\]
Consequently both the change from $J_{\eta_n(t)}$ to $J_t$ and the
change from $\ind_{\{R_n\le t\}}$ to $\ind_{\{T\le t\}}$ produce
terms tending to zero in probability, even after multiplication by
their respective factors in \eqref{eq:left-expansion}. Hence
\[
 \sqrt n\bigl(V_t^n(f)-J_t\bigr)
 =4\sqrt n\,G_n\ind_{\{T\le t\}}+o_{\Pp}(1).
\]
The conditional law of $\sqrt n\,G_n$ given $T$ is $N(0,1)$ for every
$n$, which proves \eqref{eq:left-example}.

Here $\theta_{T-}=1$, $\theta_T=2$, and $f'(\Delta X_T)=4$.
The corrected mark therefore has the law of
\[
 4\left(\sqrt U\,\Psi^-+\sqrt{1-U}\,\Psi^+\right)\sim N(0,16).
\]
Using $\theta_T$ instead would multiply this mark by $1/\sqrt2$,
giving $N(0,8)$. Since $\Pp(T\le t)=1-e^{-t}>0$, these yield
different mixture laws at time $t$. Thus the post-jump sampling value
fails even for convergence in law, and hence also for stable
convergence.
\end{proof}

The distinction comes from the recursion \eqref{eq:grid}: the length
of the interval containing $T$ is fixed at its left endpoint.
The change in $\theta$ at $T$ can affect subsequent intervals only.
In general, the sampled value at the left endpoint approaches
$\theta_{T-}$; the post-jump Brownian coefficient in \eqref{eq:R}
nevertheless remains $\sigma_T$.

\subsection{Failure under mere positivity and adaptedness}

\begin{proposition}[Failure under mere positivity and adaptedness]
\label{prop:unbounded-reciprocal}
Let $X=W$, let $f(x)=x^4$, and let the deterministic sampling intensity
be
\begin{equation}\label{eq:degenerate-theta}
 \theta_s=
 \begin{cases}
  1,&s=0,\\
  s^{3/4},&0<s\le1,\\
  1,&s>1.
 \end{cases}
\end{equation}
Then $\theta$ is positive, adapted, and bounded, and its grid is
locally finite with mesh tending to zero. Nevertheless, for every
deterministic $t>0$ and every $\varepsilon>0$,
\begin{equation}\label{eq:degenerate-failure}
 \liminf_{n\to\infty}\Pp\bigl(\sqrt n\,V_t^n(f)>\varepsilon\bigr)
 \ge\Pp(Z^4>\varepsilon)>0,\qquad Z\sim N(0,1).
\end{equation}
In particular, the proposed convergence to the jump limit
$Z_t(f)=0$ fails.
\end{proposition}

\begin{proof}
The first two grid intervals satisfy
\[
 \tau_1^n=\frac1n,\qquad
 \tau_2^n-\tau_1^n
 =\frac1{n(1/n)^{3/4}}=n^{-1/4}.
\]
Since $\theta_s\le1$, every step is at least $1/n$, and
$\tau_i^n\ge i/n$, proving local finiteness. For $i\ge1$, the
inequality $\tau_i^n\ge1/n$ implies
$\theta_{\tau_i^n}\ge n^{-3/4}$. Thus
\[
 \sup_{i\ge0}(\tau_{i+1}^n-\tau_i^n)\le n^{-1/4}\longrightarrow0.
\]
For fixed $t>0$, the second interval is completed before $t$ for all
sufficiently large $n$. Put
\[
 Q_n=\sqrt n\left(W_{\tau_2^n}-W_{\tau_1^n}\right)^4.
\]
The Brownian increment has variance $n^{-1/4}$, and
$\sqrt n\,(n^{-1/4})^2=1$, so $Q_n$ has the law of $Z^4$ for every
$n$. All summands of $V_t^n(f)$ are nonnegative; hence
$\sqrt n\,V_t^n(f)\ge Q_n$ for all sufficiently large $n$. This proves
\eqref{eq:degenerate-failure}. Brownian motion has no jumps, so both
$F_t(f)$ and the proposed limit \eqref{eq:Z} vanish. Convergence in
law to zero would imply convergence in probability to zero, contrary
to \eqref{eq:degenerate-failure}.
\end{proof}

In \eqref{eq:degenerate-theta}, $1/\theta$ is unbounded on every
interval $[0,H]$ with $H>0$, and $\theta$ is not right-continuous at
zero. This example therefore isolates the failure of a claim based
only on positivity and adaptedness. The next example shows that the
same obstruction persists when c\`adl\`ag regularity is retained.

\subsection{Failure with a strictly positive c\`adl\`ag intensity}

\begin{proposition}[The reciprocal bound cannot simply be omitted]
\label{prop:cadlag-reciprocal}
Let $X=W$, let $f(x)=x^4$, and define the deterministic process
\begin{equation}\label{eq:cadlag-theta}
 \theta_s=
 \begin{cases}
  (1-s)^3,&0\le s<1,\\
  1,&s\ge1.
 \end{cases}
\end{equation}
All the hypotheses of Theorem~\ref{thm:main} hold except for the
local boundedness of $1/\theta$. The grid is locally finite, and the
partition of $[0,1]$ formed by its completed intervals and the added
terminal point $1$ has mesh at most $n^{-1/4}$. Nevertheless, for
every $\varepsilon>0$,
\begin{equation}\label{eq:cadlag-failure}
 \liminf_{n\to\infty}
 \Pp\bigl(\sqrt n\,V_1^n(f)>\varepsilon\bigr)
 \ge\Pp(Z^4>64\varepsilon)>0,\qquad Z\sim N(0,1).
\end{equation}
Thus the fixed-time convergence to the jump limit $Z_1(f)=0$ fails.
\end{proposition}

\begin{proof}
The process $\theta$ is adapted, satisfies $0<\theta_s\le1$ for
every $s$, and has c\`adl\`ag paths. At the only discontinuity,
$\theta_{1-}=0$ and $\theta_1=1$. Hence
$\sup_{0\le s\le1}1/\theta_s=\infty$. The assumptions on $X$ and
$f$ hold as in Proposition~\ref{prop:unbounded-reciprocal}.
Every step is at least $1/n$, so $\tau_i^n\ge i/n$ and the grid
has no finite accumulation point.

Put $r_n=n^{-1/4}$. For $n>1$, let
\[
 k_n=\min\{i:\tau_i^n\ge1-r_n\},\qquad
 a_n=\tau_{k_n-1}^n,\qquad b_n=\tau_{k_n}^n,
 \qquad h_n=b_n-a_n.
\]
The index is finite, and $a_n<1-r_n\le b_n$. With $d_n=1-a_n$,
the sampling recursion gives
\[
 d_n>r_n,\qquad
 h_n=\frac1{n d_n^3}<\frac1{n r_n^3}=r_n.
\]
It follows that $b_n=a_n+h_n<1$, so this interval is already
completed at the time under consideration. Since $b_n\ge1-r_n$,
we also have $d_n-h_n\le r_n$ and therefore $d_n<2r_n$. Consequently,
\begin{equation}\label{eq:cadlag-long-interval}
 \frac{r_n}{8}<h_n<r_n,\qquad b_n<1,\qquad
 \sqrt n\,h_n^2>\frac1{64}.
\end{equation}
All grid times are deterministic, so
$W_{b_n}-W_{a_n}$ has the law of $\sqrt{h_n}\,Z$. Nonnegativity of
the summands yields
\[
 \sqrt n\,V_1^n(f)\ge\sqrt n\,(W_{b_n}-W_{a_n})^4.
\]
The term on the right has the law of $\sqrt n\,h_n^2Z^4$;
\eqref{eq:cadlag-long-interval} proves \eqref{eq:cadlag-failure}.
Since $X$ is continuous, $F_1(f)=Z_1(f)=0$. The probability bound
rules out convergence in probability to zero, and hence convergence
in law or stable convergence to that constant.

For completeness, consider any completed interval with left endpoint
$a<1$, length $h$, and $d=1-a$. Then
$h=1/(nd^3)\le d$, implying $d\ge r_n$ and $h\le r_n$.
If $\eta_n(1)<1$, its next grid point is larger than $1$. With
$d=1-\eta_n(1)$, this gives $1/(nd^3)>d$, so $d<r_n$.
If $\eta_n(1)=1$, the terminal remainder is zero. Thus
\[
 \max\left\{
   \max_{1\le i\le N_n(1)}(\tau_i^n-\tau_{i-1}^n),\,
   1-\eta_n(1)\right\}\le n^{-1/4},
\]
which proves the asserted mesh bound.
\end{proof}

\begin{remark}[Interpretation of the examples]\label{rem:reciprocal-scope}
Proposition~\ref{prop:left-scale} determines the left-limit scale
in \eqref{eq:R}. The other two examples concern the disappearance of
the continuous contribution: a Brownian increment on an interval of
length $h$ contributes a fourth power of order $h^2$, which need not
vanish after multiplication by $\sqrt n$. The local reciprocal bound
in \eqref{eq:theta} excludes the exhibited intervals of length
comparable to $n^{-1/4}$ by giving steps of order $1/n$ after
localization. Proposition~\ref{prop:unbounded-reciprocal} also lacks
right continuity, whereas Proposition~\ref{prop:cadlag-reciprocal}
retains every other hypothesis of the theorem. The latter proves
that the reciprocal bound cannot be deleted without a replacement
assumption; it does not assert that every model with unbounded
$1/\theta$ fails to converge. Only bounds on each finite horizon
are required, not a uniform bound over all $s\ge0$.
\end{remark}

\section{Joint stable convergence of the local marks}\label{sec:marks}

We first reduce to bounded coefficients and record the geometry of the
sampling intervals. We then establish joint stable convergence of the
local positions and Brownian increments using the information available
at each left endpoint.

\subsection{Localization and the sampling intervals}\label{sec:localization}

For the proof on a fixed horizon we work under bounded versions of
the hypotheses. The standard localization from (H) to (SH)
\cite[Lemma 4.4.9]{JP}, together with clipping of $\theta$, reduces the
proof to the case in which, for deterministic constants $K,C,c>0$,
\begin{equation}\label{eq:bounded}
 \norm{b_s}+\norm{\sigma_s}+\norm{X_s}\le K,\qquad
 \norm{\delta(s,z)}\le\Gamma(z)\le K,\qquad
 \int_E\Gamma^2\,d\lambda\le K,
 \qquad c\le\theta_s\le C.
\end{equation}
The reduction is local for the pair $(X,\theta)$: if two such pairs
coincide before a stopping time, their recursively constructed grids
and observed functionals coincide before that time. Their limit kernels
also coincide there. The exit probabilities tend to zero as the
localization levels increase, by \eqref{eq:theta} and (H). Stable
convergence therefore transfers back by testing against bounded $B$ and
bounded continuous functions. Enlarging the localization horizon slightly
also covers the interval immediately after any time being considered.
We may replace $f$ outside a fixed ball containing all increments and
jumps by a compactly supported $C^2$ function.

Write $a_i^n=\tau_{i-1}^n$, $b_i^n=\tau_i^n$, and
$h_i^n=b_i^n-a_i^n$. Under \eqref{eq:bounded},
\begin{equation}\label{eq:steps}
 \frac1{nC}\le h_i^n\le\frac1{nc}.
\end{equation}
Each $b_i^n$ is $\F_{a_i^n}$-measurable. The grid times are stopping
times; moreover the positive grid times are predictable, announced by
$a_i^n+(1-1/k)h_i^n$ as $k\to\infty$. Hence they almost surely avoid
the atoms of any finite-intensity restriction of $p$. The same is true
for all atoms capable of producing a nonzero jump of $X$, by a countable
union over $\{\Gamma>1/m\}$.

For a time $S$ which is not a grid point set
\[
 \ell_n(S)=a_i^n,\qquad r_n(S)=b_i^n
 \quad\text{when }a_i^n<S<b_i^n.
\]
Then, pathwise at every $S>0$ of this kind,
\begin{equation}\label{eq:geometry}
 \ell_n(S)<S<r_n(S),\qquad
 \ell_n(S)\to S,\quad r_n(S)\to S,\qquad
 n(r_n(S)-\ell_n(S))\to\theta_{S-}^{-1}.
\end{equation}
The last limit follows from
$n(r_n(S)-\ell_n(S))=1/\theta_{\ell_n(S)}$ and the existence of the
left limit of $\theta$.

\subsection{The joint mark limit}

The following argument is the part of the proof that handles the
dependence of the grid on the observed process. It conditions at the
left endpoint of an interval, where its length is already known.
The limit has the same uniform--Gaussian mark structure as the
regular-grid result in \cite[Proposition 4.4.10]{JP}, with the additional
sampling factor $\theta_{S-}^{-1/2}$. The proof below establishes this
structure for the present endogenous grid.
\begin{lemma}\label{lem:marks}
Let $A\subset E$ be deterministic with $0<\lambda(A)<\infty$, and let
$S_1<S_2<\cdots$ be the successive atoms of $p(\cdot,A)$. Assume
\eqref{eq:bounded} for $\theta$. Define
\begin{align*}
 U_n(S)&=\frac{S-\ell_n(S)}{r_n(S)-\ell_n(S)},\\
 B_n^-(S)&=\sqrt n\,(W_S-W_{\ell_n(S)}),\qquad
 B_n^+(S)=\sqrt n\,(W_{r_n(S)}-W_S).
\end{align*}
For every finite number of atoms, jointly and stably relative to $\F$,
\begin{equation}\label{eq:marklimit}
 \bigl(U_n(S_j),B_n^-(S_j),B_n^+(S_j)\bigr)_j
 \st
 \left(U_j,\sqrt{\frac{U_j}{\theta_{S_j-}}}\Psi_j^-,
            \sqrt{\frac{1-U_j}{\theta_{S_j-}}}\Psi_j^+\right)_j.
\end{equation}
The auxiliary triples are independent across $j$ and independent of $\F$.
\end{lemma}

\begin{proof}
Write $\lambda_A=\lambda(A)$ and $N_t=p((0,t]\times A)$.
Fix a horizon $H$ and a bounded continuous function
$\varphi$ on $[0,1]\times\R^{d'}\times\R^{d'}$. For $v>0$ put
\begin{equation}\label{eq:kernel}
 K_\varphi(v)=\int_0^1
 \E\!\left[\varphi\bigl(u,\sqrt{vu}\,G^-,
                           \sqrt{v(1-u)}\,G^+\bigr)\right]du,
\end{equation}
where $G^-,G^+$ are independent standard Gaussian vectors.
This is a bounded continuous function of $v$.

\emph{Step 1: a centered interval identity.}
Suppress the superscript $n$ on interval endpoints. Set
\begin{equation}\label{eq:zeta}
 \zeta_i^n=\sum_{S\in(a_i,b_i]:\Delta N_S=1}
 \left\{\varphi\bigl(U_n(S),B_n^-(S),B_n^+(S)\bigr)
             -K_\varphi(nh_i)\right\}.
\end{equation}
Conditionally on $\F_{a_i}$, $h_i$ is known and the future Brownian
motion and Poisson process have their standard independent laws.
Integration over a Poisson atom in $(a_i,a_i+h_i]$, followed by the
change of variable $s=a_i+u h_i$, gives
\begin{equation}\label{eq:center}
 \E[\zeta_i^n\mid\F_{a_i}]=0.
\end{equation}
This identity counts all atoms in the interval; it does not require that
there be exactly one atom.

\emph{Step 2: testing against the original sigma-field.}
We establish the following weighted identity
\begin{equation}\label{eq:weighted}
 \E\!\left[B\sum_{i:b_i\le H}H_i^n\zeta_i^n\right]\longrightarrow0
\end{equation}
for every bounded $\F$-measurable $B$ and every array $H_i^n$ which is
$\F_{a_i}$-measurable and bounded in absolute value by one.
Let $M_s=\E[B\mid\F_s]$ be its bounded c\`adl\`ag martingale version.
Using \eqref{eq:center}, the left side of \eqref{eq:weighted} is
\[
 \E\sum_{i:b_i\le H}H_i^n\zeta_i^n(M_{b_i}-M_{a_i}).
\]
There are finitely many atoms of $N$ on $[0,H]$. For each such atom $S$,
the c\`adl\`ag property and \eqref{eq:steps} give
\[
 M_{r_n(S)}-M_{\ell_n(S)}-\Delta M_S\longrightarrow0
 \quad\text{almost surely}.
\]
The absolute error in replacing the martingale increment by $\Delta M_S$
in the preceding expectation is bounded by a constant times $N_H$.
It therefore tends to zero in $L^1$, uniformly in the choices of $H_i^n$.

For $a_i<s\le b_i$ define
\begin{align}\label{eq:q}
 q_n(s)={}&\E_G\!\left[
 \varphi\left(\frac{s-a_i}{h_i},\sqrt n(W_s-W_{a_i}),
                                  \sqrt{n(b_i-s)}\,G\right)\right]
 -K_\varphi(nh_i).
\end{align}
At a Poisson time $S$, conditioning on $\F_S$ integrates out the
post-$S$ Brownian increment and replaces the braces in \eqref{eq:zeta}
by $q_n(S)$. The right endpoint $r_n(S)$ is already $\F_S$-measurable,
which justifies this conditional Gaussian calculation.
The process $q_n$, defined on left-open, right-closed grid intervals,
is predictable: $W$ is continuous and the endpoint data on each interval
are known at its left endpoint. It is uniformly bounded.

There exists a predictable process $a_s$, with
$|a_s|\le2\lambda_A\norm{B}_\infty$, such that for every bounded
predictable $v$ supported in $[0,H]$,
\begin{equation}\label{eq:projection}
 \E\sum_{S\le H:\Delta N_S=1}v_S\Delta M_S
       =\E\int_0^H v_s a_s\,ds.
\end{equation}
Indeed, on the predictable sigma-field the signed measure on the left
is dominated in absolute value by
$2\norm{B}_\infty\lambda_A\,d\Pp\,ds$; the Radon--Nikodym theorem
gives \eqref{eq:projection} and its bound.
Consequently the expectation in \eqref{eq:weighted}, up to an $o(1)$
term, equals
\begin{equation}\label{eq:projected}
 \E\sum_{i:b_i\le H} H_i^n\int_{a_i}^{b_i}q_n(s)a_s\,ds.
\end{equation}

Approximate $a$ in $L^1(d\Pp\,ds)$ on $[0,H]$ by bounded elementary
predictable processes, constant on finitely many deterministic intervals
$(t_j,t_{j+1}]$, with each value $\F_{t_j}$-measurable. The error in
\eqref{eq:projected} is bounded by
$2\norm{\varphi}_\infty$ times this $L^1$ error, uniformly in $n$ and
$H_i^n$. For an interval $(a_i,b_i]$ wholly inside one such deterministic
interval, its elementary value is $\F_{a_i}$-measurable and
\begin{equation}\label{eq:timeaverage}
 \E\!\left[\int_{a_i}^{b_i}q_n(s)\,ds\,\middle|\,\F_{a_i}\right]=0.
\end{equation}
This follows directly from \eqref{eq:kernel}: integrate the two independent
Brownian pieces and then integrate the relative position over $[0,1]$.
The grid intervals crossing any of the finitely many $t_j$ have total
length at most a constant times $1/n$. Their contribution vanishes.
First let $n\to\infty$, then remove the elementary approximation.
This proves \eqref{eq:weighted}.

\emph{Step 3: joint convergence and independence of the marks.}
The only interval crossing $H$ can contain an atom only in
$(H-1/(nc),H]$, an event of probability at most $\lambda_A/(nc)$.
On each path, all atoms on $[0,H]$ are in distinct grid intervals for
all sufficiently large $n$. Furthermore, for each atom $S$,
\begin{equation}\label{eq:kernelreplace}
 K_\varphi\bigl(n(r_n(S)-\ell_n(S))\bigr)
       \longrightarrow K_\varphi(\theta_{S-}^{-1}).
\end{equation}

To prove joint convergence, consider product test functions. On $\{N_H=m\}$ choose bounded
continuous functions $\varphi_1,\ldots,\varphi_m$, scaled to have
absolute value at most one. Write
$\varphi_j^n=\varphi_j(U_n(S_j),B_n^-(S_j),B_n^+(S_j))$
and $K_j=K_{\varphi_j}(\theta_{S_j-}^{-1})$. In the telescoping identity
\[
 \prod_{j=1}^m\varphi_j^n-\prod_{j=1}^m K_j
 =\sum_{j=1}^m
   \left(\prod_{q<j}\varphi_q^n\right)
   (\varphi_j^n-K_j)\left(\prod_{q>j}K_q\right),
\]
use \eqref{eq:weighted} with original multiplier
$B\ind_{\{N_H=m\}}\prod_{q>j}K_q$.
In the array $H_i^n$, select $N_{a_i}=j-1$ and multiply the marks of
the preceding $j-1$ atoms, all of whose observation intervals have
ended by $a_i$. This is $\F_{a_i}$-measurable. Once the atoms occupy
distinct intervals it equals exactly the required preceding product.
Use \eqref{eq:kernelreplace} to replace the interval kernel by $K_j$.
The exceptional multiple-atom intervals have vanishing contribution,
by domination by a constant times $N_H$ and eventual pathwise separation.
It follows that
\begin{equation}\label{eq:product}
 \E\!\left[B\ind_{\{N_H=m\}}\prod_{j=1}^m\varphi_j^n\right]
 \longrightarrow
 \E\!\left[B\ind_{\{N_H=m\}}\prod_{j=1}^m K_j\right].
\end{equation}
The right side is the conditional product kernel of the marks displayed
in \eqref{eq:marklimit}.

For completeness, these tests determine joint stable convergence.
The $U_n$ are bounded, and a conditional Brownian/Poisson calculation gives
\[
 \sup_n\E\sum_{S\le H:\Delta N_S=1}
     \bigl(\norm{B_n^-(S)}^2+\norm{B_n^+(S)}^2\bigr)<\infty.
\]
One may include the crossing interval in this bound and use
$h_i\le1/(nc)$ and at most $nCH+1$ intervals. Thus the finite mark
vectors are tight. Finite products of continuous functions determine
their joint laws; applying \eqref{eq:product} also to nonnegative
bounded multipliers and normalizing the resulting finite measures gives
stable convergence. The Poisson tail removes the restriction to bounded
$m$. Finally $H\to\infty$ covers any fixed finite number of atoms.
\end{proof}

\section{Some key lemmas}\label{sec:lemmas}

The first lemma transfers the Brownian mark limit to the increment of
$X$ surrounding an isolated jump. The second gives a uniform estimate
that permits removal of a cutoff near zero, including in the presence
of infinitely many jumps.

\subsection{Local increments of the semimartingale}

The following lemma adapts the local increment approximation in
\cite[Proposition 4.4.10, equations (4.4.19)--(4.4.23)]{JP} to the
recursive grid. We give the argument to verify that the filtration
enlargement and the increment estimates remain valid at its random
endpoints.

\begin{lemma}\label{lem:localX}
Assume \eqref{eq:bounded}. Let $A_m=\{z:\Gamma(z)>1/m\}$, and let
$S$ be an atom of $p(\cdot,A_m)$. On each fixed finite horizon,
\begin{align}\label{eq:localX}
 &\sqrt n\bigl(X_{r_n(S)}-X_{\ell_n(S)}-\Delta X_S\bigr)\notag\\
 &\hspace{1cm}
 =\sigma_{S-}B_n^-(S)+\sigma_SB_n^+(S)+o_{\Pp}(1).
\end{align}
The remainders tend to zero jointly for any finite collection of these
atoms. Accordingly the left sides converge jointly and stably to the
marks \eqref{eq:R}, indexed by these atoms, including atoms for which
$\Delta X_S=0$.
\end{lemma}

\begin{proof}
Set
\begin{align*}
 J^{(m)}&=(\delta\ind_{A_m})\star p,\qquad Y^{(m)}=X-J^{(m)},\\
 b_s^{(m)}&=b_s-\int_{A_m}\delta(s,z)
                        \ind_{\{\norm{\delta(s,z)}\le1\}}\lambda(dz).
\end{align*}
For $m\ge1$,
\[
 Y_t^{(m)}=X_0+\int_0^t b_s^{(m)}\,ds
       +\int_0^t\sigma_s\,dW_s
       +(\delta\ind_{A_m^c})\star(p-q)_t.
\]
The drift $b^{(m)}$ is bounded for fixed $m$, since
$\int_{A_m}\Gamma\,d\lambda\le m\int\Gamma^2\,d\lambda$.

Enlarge the filtration initially by the full restriction $p|_{A_m}$.
In the resulting filtration $(\G_s)$, $W$ is still Brownian and
$p|_{A_m^c}$ still has compensator $ds\,\lambda|_{A_m^c}(dz)$.
This follows from the joint independent increments of the drivers;
it is also the finite-intensity enlargement used in
\cite[Proposition 2.1.10]{JP}.
The time $S$ is now $\G_0$-measurable. Moreover, $\ell_n(S)$ is
a $\G$-stopping time: starting from $\tau_i^n$, the test
$\tau_{i+1}^n>S$ is $\G_{\tau_i^n}$-measurable because the next endpoint
is known there. The first index passing that test therefore selects a
stopping time. The same holds for $r_n(S)$.

Eventually the interval containing $S$ contains no other $A_m$-atom.
On this event its increment after removal of $\Delta X_S$ is the
increment of $Y^{(m)}$. Its drift, multiplied by $\sqrt n$, is bounded
by $C_m/\sqrt n$.

To bound its compensated-jump increment, put $u_n=n^{-1/4}$.
For marks with $\Gamma\le u_n$, the second moment after multiplication
by $\sqrt n$ is at most
\[
 C\int_{\{\Gamma\le u_n\}}\Gamma^2\,d\lambda\longrightarrow0.
\]
For the remaining marks in $A_m^c$, the probability of an atom in the
interval is at most
\[
 \frac Cn\lambda(\Gamma>u_n)
 \le Cn^{-1/2}\int\Gamma^2\,d\lambda\longrightarrow0.
\]
Their compensator, after multiplication by $\sqrt n$, has norm at most
\[
 \frac C{\sqrt n}\int_{\{\Gamma>u_n\}}\Gamma\,d\lambda
 \le Cn^{-1/4}\int\Gamma^2\,d\lambda\longrightarrow0.
\]
All three bounds apply to stopping intervals in the enlarged filtration
and use only their deterministic upper length bound $C/n$.

For the Brownian part before $S$, first freeze at the adapted value
$\sigma_{\ell_n(S)}$. The isometry bounds the squared error, multiplied
by $n$, by
\[
 C\E\!\left[
 \sup_{\ell_n(S)\le u<S}
 \norm{\sigma_u-\sigma_{\ell_n(S)}}^2\right]\longrightarrow0.
\]
The assertion is restricted to $\{S\le H\}$ when necessary; this event
is $\G_0$-measurable. Boundedness and the left limit of $\sigma$ give
the convergence. Next $\sigma_{\ell_n(S)}\to\sigma_{S-}$, and
$B_n^-(S)$ is tight by Lemma~\ref{lem:marks}.
After $S$, freeze directly at $\sigma_S$; right continuity and the
same isometry give the corresponding vanishing error. This proves
\eqref{eq:localX}. Joint convergence follows for finitely many atoms.
Lemma~\ref{lem:marks} and stable multiplication by the original
coefficients then give the final assertion.
\end{proof}

\subsection{A uniform estimate for the small-function remainder}

The following estimate uses the It\^o decomposition underlying the
small-jump argument in \cite[proof of Theorem 5.1.2]{JP}.
We formulate it as a uniform bound for functions with a small Hessian
and verify it for the completed intervals of the recursive grid.
No finite-activity assumption is needed. Constants
may depend on the fixed horizon, the dimension, and the bounds in
\eqref{eq:bounded}, but not on $n$, $g$, or the constant $a$ below.

\begin{lemma}\label{lem:small}
Suppose \eqref{eq:bounded} holds. If $g\in C^2(\R^d;\R)$ satisfies
$g(0)=0$, $Dg(0)=0$, and
\begin{equation}\label{eq:smallass}
 \norm{D^2g(x)}\le a\norm{x}\quad\text{for every }x,
\end{equation}
then, for each fixed $t$, uniformly in $n\ge1$,
\begin{equation}\label{eq:smallbound}
 \E\left|\sqrt n\left(
      V_t^n(g)-F_{\eta_n(t)}(g)\right)\right|\le C_t a.
\end{equation}
The same bound, with a possibly larger constant, holds with $F_t(g)$.
\end{lemma}

\begin{proof}
Taylor's formula implies
\begin{equation}\label{eq:smallfirst}
 |g(x)|\le\frac a6\norm{x}^3,\qquad
 \norm{Dg(x)}\le\frac a2\norm{x}^2.
\end{equation}
Define
\[
 K_g(x,y)=g(x+y)-g(x)-g(y),\qquad
 H_g(x,y)=K_g(x,y)-Dg(x)y.
\]
The following two bounds will be used:
\begin{align}\label{eq:kernelbounds}
 |K_g(x,y)|&\le C a\bigl(\norm{x}^2\norm{y}
                                +\norm{x}\norm{y}^2\bigr),\notag\\
 |H_g(x,y)|&\le C a\norm{x}\norm{y}^2.
\end{align}
For the first, integrate $x^\top D^2g(ux+vy)y$ over $(u,v)\in[0,1]^2$.
For the second, if $\norm{y}\le\norm{x}$, expand $g(x+y)$ about $x$
and bound $g(y)$ using \eqref{eq:smallfirst}. If
$\norm{x}<\norm{y}$, instead expand $g(y+x)$ about $y$ and bound
$g(x)$ and $Dg(x)y$. In this latter case all terms are bounded by a
constant times $a\norm{x}\norm{y}^2$. This proves both inequalities
without assuming a third derivative of $g$.

Under \eqref{eq:bounded}, all jumps can be compensated by using the
bounded drift
\[
 \bar b_s=b_s+\int_{\{\norm{\delta(s,z)}>1\}}\delta(s,z)\lambda(dz).
\]
On a grid interval $(a_i,b_i]$ let $Y_s^i=X_s-X_{a_i}$.
It\^o's formula, subtracting the sum of $g(\Delta X)$ on that interval,
decomposes the unscaled error into a finite-variation term and a
square-integrable martingale term, with integrands
\begin{align}\label{eq:ItoA}
 dA_s^i={}&\left(Dg(Y_{s-}^i)\bar b_s
       +\tfrac12\operatorname{tr}\bigl(D^2g(Y_{s-}^i)c_s\bigr)
       +\int_E H_g(Y_{s-}^i,\delta(s,z))\lambda(dz)\right)ds,\\
 dM_s^i={}&Dg(Y_{s-}^i)\sigma_s\,dW_s
       +\int_E K_g(Y_{s-}^i,\delta(s,z))(p-q)(ds,dz),\label{eq:ItoM}
\end{align}
where $c_s=\sigma_s\sigma_s^\top$.
The stochastic Brownian integrals can equivalently use $\sigma_{s-}$;
the two versions agree in $ds\,d\Pp$ measure.

Restricting the sum to completed intervals preserves the
martingale property: $\ind_{\{b_i\le t\}}$ is $\F_{a_i}$-measurable,
so the integrands in \eqref{eq:ItoM}, multiplied by these indicators
and restricted to $(a_i,b_i]$, are predictable.

We record the moment bounds needed to sum these integrals. The
conditional maximal inequalities for the bounded drift, Brownian
integral, and compensated jump integral give, for $q=2,4$ and a
bounded $\F_{a_i}$-measurable duration $h_i$,
\begin{align*}
 \E\left[\sup_{a_i\le s\le b_i}\norm{Y_s^i}^2
                    \,\middle|\,\F_{a_i}\right]&\le C h_i,\\
 \E\left[\sup_{a_i\le s\le b_i}\norm{Y_s^i}^4
                    \,\middle|\,\F_{a_i}\right]&\le C(h_i+h_i^2+h_i^4)
                                                        \le C h_i.
\end{align*}
The last constant uses $h_i\le1/c$; $\int\Gamma^4d\lambda<\infty$
follows from boundedness of $\Gamma$. Consequently, with the sum always
restricted to $b_i\le t$,
\begin{align}\label{eq:moments}
 \E\sum_i\int_{a_i}^{b_i}
   \bigl(\norm{Y_{s-}^i}^2+\norm{Y_{s-}^i}^4\bigr)ds
   &\le C\E\sum_i h_i^2\le\frac{C_t}{n},\notag\\
 \E\sum_i\int_{a_i}^{b_i}\norm{Y_{s-}^i}\,ds
   &\le C\E\sum_i h_i^{3/2}\le\frac{C_t}{\sqrt n}.
\end{align}
Multiplication of the conditional bounds by $\ind_{\{b_i\le t\}}$
is legitimate because that event is known at $a_i$.

Equations \eqref{eq:smallfirst}, \eqref{eq:kernelbounds}, and
\eqref{eq:moments} show that the expected total variation of the
sum in \eqref{eq:ItoA}, multiplied by $\sqrt n$, is at most $C_t a$.
The second moment of the scaled sum in \eqref{eq:ItoM} is at most
\begin{align*}
 C n a^2\E\sum_i\int_{a_i}^{b_i}
 \left[\norm{Y_{s-}^i}^4
       +\norm{Y_{s-}^i}^4\int_E\Gamma^2\,d\lambda
       +\norm{Y_{s-}^i}^2\int_E\Gamma^4\,d\lambda\right]ds
 \le C_t a^2.
\end{align*}
We used the isometries and orthogonality of continuous and purely
discontinuous martingales. Cauchy--Schwarz proves \eqref{eq:smallbound}.

Finally, $\eta_n(t)\ge(t-1/(nc))\vee0$. The compensator formula and
\eqref{eq:smallfirst} give
\begin{equation}\label{eq:endpointg}
 \E\left[\sqrt n\,
   |F_t(g)-F_{\eta_n(t)}(g)|\right]
 \le\frac{C a}{\sqrt n}\int_E\Gamma^3\,d\lambda.
\end{equation}
This proves the last assertion.
\end{proof}

\section{Proof of the theorem}\label{sec:proof}

\begin{proof}[Proof of Theorem~\ref{thm:main}]
By the localization argument in Section~\ref{sec:localization}, it
suffices to work under \eqref{eq:bounded}. The proof follows the
standard truncation argument in \cite[Theorem 5.1.2]{JP}, using the
adaptive-grid lemmas established above, and proceeds in three steps.

\emph{Step 1: existence and tails of the limit.}
Integrating \eqref{eq:f} twice shows that
\begin{equation}\label{eq:orders}
 f(x)=o(\norm{x}^3),\qquad Df(x)=o(\norm{x}^2).
\end{equation}
On every finite horizon the jumps of a semimartingale have a square
summable sequence of sizes, with only finitely many sizes exceeding
any positive threshold. Thus $F_t(f)$ is absolutely convergent.
Also, conditionally on $\F$, the expected norm of $R_p^{\mathrm{corr}}$
is bounded uniformly over $T_p\le t$. Equation \eqref{eq:orders} gives
\[
 \sum_{p:T_p\le t}\widetilde\E\left[
      |Df(\Delta X_{T_p})R_p^{\mathrm{corr}}|\,\middle|\,\F\right]
 <\infty\quad\text{almost surely}.
\]
Conditional Tonelli therefore proves absolute convergence of
\eqref{eq:Z}. The same reasoning works without localization, using the
pathwise bounds on $\sigma$, $\theta^{-1}$, and the jump sizes.

The summands are conditionally centered and independent. For any
function $g$ as in Lemma~\ref{lem:small},
\begin{equation}\label{eq:limittail}
 \widetilde\E|Z_t(g)|^2
 \le C a^2\E\sum_{s\le t}\norm{\Delta X_s}^4
 \le C_t a^2\int_E\Gamma^4\,d\lambda.
\end{equation}

\emph{Step 2: functions vanishing near zero.}
Suppose first that $f(x)=0$ for $\norm{x}\le\varepsilon$.
Choose a fixed $m$ with $2/m<\varepsilon$, and use $Y^{(m)}$ from
Lemma~\ref{lem:localX}. Its jumps have norm at most $1/m$.
The elementary oscillation property of a c\`adl\`ag path implies
\[
 \limsup_{h\downarrow0}\sup_{\substack{0\le u\le v\le t+1\\v-u\le h}}
       \norm{Y_v^{(m)}-Y_u^{(m)}}\le\frac1m.
\]
Together with \eqref{eq:steps}, this shows that all its grid increments
have norm at most $2/m$ for all sufficiently large $n$, almost surely.
The finitely many $A_m$-atoms on this horizon are eventually in distinct
grid intervals.

On this event an interval without an $A_m$-atom contributes zero.
All jumps outside $A_m$ also have zero $f$-value. For an $A_m$-atom $S$,
put
\[
 Q_n(S)=\sqrt n\bigl(X_{r_n(S)}-X_{\ell_n(S)}-\Delta X_S\bigr).
\]
Lemma~\ref{lem:localX} gives joint tightness and stable convergence of
these variables to their corrected marks. Taylor's formula yields
\begin{equation}\label{eq:Taylor}
 \sqrt n\{f(\Delta X_S+n^{-1/2}Q_n(S))-f(\Delta X_S)\}
       =Df(\Delta X_S)Q_n(S)+o_{\Pp}(1).
\end{equation}
This holds jointly for any fixed number of atoms. Their random finite
number is handled by first restricting it to a deterministic bound and
then using the Poisson tail.
An atom before $t$ can be omitted by the completed-interval sum only if
it lies in $(t-1/(nc),t]$, an event of probability at most
$\lambda(A_m)/(nc)$. Hence \eqref{eq:Taylor}, Lemma~\ref{lem:marks}, and
stable multiplication by $Df(\Delta X_S)$ imply \eqref{eq:main} for
functions vanishing near zero. The limit is precisely \eqref{eq:Z},
because the other jumps have zero derivative as well.

\emph{Step 3: removal of the cutoff.}
Choose a smooth compactly supported $\chi:\R^d\to[0,1]$ equal to one
on the unit ball and zero outside the ball of radius two. Define
\[
 g_\varepsilon(x)=f(x)\chi(x/\varepsilon),\qquad
 f_\varepsilon=f-g_\varepsilon.
\]
The function $f_\varepsilon$ vanishes on the ball of radius
$\varepsilon$. If
\[
 \omega(r)=\sup_{0<\norm{x}\le r}
          \frac{\norm{D^2f(x)}}{\norm{x}},
\]
then $\omega(r)\to0$. Differentiation of the cutoff, using
$\norm{Df(x)}\le\omega(2\varepsilon)\norm{x}^2/2$ and
$|f(x)|\le\omega(2\varepsilon)\norm{x}^3/6$ on its support, gives
\begin{equation}\label{eq:cutoffbound}
 \norm{D^2g_\varepsilon(x)}\le a_\varepsilon\norm{x}
 \quad\text{for every }x,\qquad
 a_\varepsilon=C_\chi\omega(2\varepsilon)\longrightarrow0.
\end{equation}
Write $Z_n(t;h)=\sqrt n(V_t^n(h)-F_t(h))$.
Lemma~\ref{lem:small} and \eqref{eq:limittail} imply
\begin{align}\label{eq:together}
 \sup_n\E|Z_n(t;f)-Z_n(t;f_\varepsilon)|
     &\le C_t a_\varepsilon\longrightarrow0,\notag\\
 \widetilde\E|Z_t(f)-Z_t(f_\varepsilon)|^2
     &\le C_t a_\varepsilon^2\longrightarrow0.
\end{align}
For fixed $\varepsilon$, Step 2 gives
$Z_n(t;f_\varepsilon)\st Z_t(f_\varepsilon)$.
Testing against a bounded $\F$-measurable $B$ and a bounded Lipschitz
function, \eqref{eq:together} lets us first send $n\to\infty$ and then
$\varepsilon\downarrow0$. These tests, together with tightness inherited
from the same approximation, give stable convergence for all bounded
continuous tests. This proves \eqref{eq:main} under the bounded hypotheses.

For finitely many deterministic times $t_1,\ldots,t_k$, apply the same
argument to the vector of statistics. Step~2 uses joint convergence
of the marks on a common horizon, while the approximation errors in
\eqref{eq:together} are bounded by summing over the $k$ coordinates.
Testing against bounded Lipschitz functions on $\R^k$ therefore gives
the finite-dimensional stable convergence. Finally, the localization
argument of Section~\ref{sec:localization} removes the boundedness
restriction and completes the proof.
\end{proof}

\end{document}